\documentclass[smallextended,referee,envcountsect]{svjour3} 
\smartqed 
\usepackage{amssymb}
\usepackage{graphicx}
\usepackage{amsmath}
\usepackage{multirow}
\usepackage{mathtools}
\usepackage[english]{babel}
\usepackage{amsmath}
\usepackage{rotating}
\usepackage{array}
\usepackage{float}
\usepackage{subcaption}

\usepackage[utf8]{inputenc}
\usepackage[english]{babel}

\usepackage{caption}

\usepackage[numbers,compress]{natbib}

\journalname{JOTA}

\begin{document}

\title{Generalized Lagrangian Jacobi-Gauss-Radau collocation method for solving a nonlinear 2-D optimal control problem with
the classical diffusion equation}
\titlerunning{ Wave Diffusion }   


\author{Kourosh Parand         \and
        Sobhan Latifi\and
         Mehdi Delkhosh\and  
        Mohammad M. Moayeri 
}

\institute{K. Parand \at
              Department of Computer Sciences, Shahid Beheshti University, G.C. Tehran, Iran.\\
Department of Cognitive Modeling, Institute for Cognitive and Brain Sciences, Shahid Beheshti University, G.C. Tehran, Iran.\\
              \email{k\_parand@sbu.ac.ir}           
           \and
           S. Latifi, M. Delkhosh, and M.M. Moayeri \at
              Department of Computer Sciences, Shahid Beheshti University, G.C. Tehran, Iran.\\ \email{s.latifi@mail.sbu.ac.ir, mehdidelkhosh@yahoo.com, m\_moayeri@sbu.ac.ir}              
  }            

\date{Received: date / Accepted: date}

\maketitle
\begin{abstract}
In this paper,  a nonlinear 2D Optimal Control Problem (2DOCP) is considered. The quadratic performance index of a nonlinear cost function  is endowed with the state and  control functions. In this problem, the dynamic constraint of the system is given by a classical diffusion equation. This article is concerned with a generalization of  Lagrangian functions. Besides, a Generalized Lagrangian Jacobi-Gauss-Radau (GLJGR)-collocation method is introduced and applied to solve the aforementioned 2DOCP. Based on initial and boundary conditions, the time and space  variables $t$  and $x$  are considered Jacobi-Gauss-Radau points clustered on first or end of interval respectively. Then, to solve the 2DOCP, Lagrange Multipliers are used and the optimal control problem is reduced to a parameter optimization problem. Numerical results demonstrate its accuracy, efficiency, and versatility of the presented method.
\end{abstract}


\section{Introduction}

In order to present 2DOCP  solved in  this manuscript, firstly, we give an introduction to the 2DOCP and provides an explanation of the functions and parameters defined in this problem. A brief review and history of these equations and spectral and Pseudospectral (PS) methods are in the following subsections.

\subsection{The governing  equations}
Optimum control problems rise in the minimization of a functional over a set of admissible control functions subject to dynamic constraints on the state and control functions \cite{agrawal2007,agrawal2006}. As the equations of dynamics in the system are reformed by a partial differential equation-- with time and space variables-- this 2DOCP is known as an optimal control of a distributed system \cite{agrawal2007}. The formulation of this optimal control problem  is \cite{meme}:
\begin{equation}\label{minproblem}
min ~~ J=\frac{1}{2}\int_0^1\int_0^R x^r\big(c_1z^2(x,t)+c_2y^2(x,t)\big)dxdt,
\end{equation}
subject to 
\begin{equation}\label{subjecto}
\frac{\partial z}{\partial t}=k(\frac{\partial^2 z}{\partial x^2}+\frac{r}{x}\frac{\partial z}{\partial t})+y(x,t),
\end{equation}
with initial and boundary   conditions
\begin{equation}\label{condition}
z(x,0)=z_0(x), 0<x<R,~~~~~~~z(R,t)=0,t>0.
\end{equation}
In fact, this is a nonlinear 2-D quadratic optimal control problem with the dynamic system of classical diffusion equation.\\
$z(x,t)$ and $y(x,t)$ are the state and control smooth functions, respectively. $c_1$ and $c_2$ are two arbitrary functions. The upper bound of variable $t$ is considered $1$.  The parameter $r$ is specified in numerical examples as $r=1$  or $r=2$.\\
The purpose of solving this problem  is to approximate the control and state functions that minimize the  $J$.   

\subsection{The literature of Optimal control problems}
Two-dimensional (2D) systems and their beneficial  applications in many different industrial fields draw the attention of scientists presently. These applications rise in heat transfer, image processing,  seismological and geophysical data processing, distributed systems, restoration of noisy images, earthquake signal processing, water stream heating gas absorption, smart materials,  and transmission lines \cite{meme2015,lewis,mars,roes}. The miscellaneous chemical, biological,  and physical problems are modeled by   diffusion processes involving control function mentioned in Eqs. (\ref{minproblem})--(\ref{condition}). By the aid of Roesser's model \cite{roes}, Attasi's model \cite{attasi1973,attasi1975} Fornasini-Marchesini's models \cite{forn1976,forn1978}, the state-space models of 2D systems are organized \cite{meme2015}. These models are used extensively to analyze controllability, stability,  observability  of  2D systems.

Remarkable studies have been done in the area of optimal controls, and excellent article are written hereby \cite{bryson,sage,agrawal1989,lotfi2,samimi}. Among these studies,  numerical techniques have been used to solve optimal control problems   \cite{agrawal1989,gre}. Moreover,  Agrawal \cite{agrawal2007}  presented a general formulation and a numerical scheme for Fractional Optimal Control for a class of distributed systems. He used eigenfunctions to develop his method. In other works, Manabe \cite{manabe}, Bode et al. \cite{bode} and Rabiee et al. \cite{rabiee} studied fractional order optimal control problems. Additionally, Mamehrashia et al. \cite{meme,meme2015} and Lotfi et al. \cite{lotfi} employed Ritz method to solve optimal control problems. With the Variational method, Yousefi et al. \cite{yusefi}  found the solution of the optimal control of linear systems, approximately. Li et al. \cite{Li} considered a continuous time 2D system and converted it to the discrete-time 2D model. In other works, Wei et al. \cite{wei} and Zhang et al. \cite{zhang} investigated an optimal control problem in continuous-time 2D Roesser's model. They employed iterative adaptive critic design algorithm and the adaptive dynamic programming method to approximate the solution. Sabeh et al. \cite{sabeh} introduced a pseudo-spectral method for the solution of distributed optimal control problem with the viscous Burgers equation

\subsection{The literature of Spectral and PS methods}
The main feature of spectral methods is to use different orthogonal polynomials/functions as trial functions. These polynomials/functions are global and infinitely differentiable. These methods are applied to 4 types of problems: periodic problems, non-periodic problems,   whole line problems and half line problems.  Trigonometric polynomials for periodic problems; classical Jacobi, ultraspherical, Chebyshev and Legendre  polynomials for non-periodic problems; Hermite polynomials for problems on the whole line; and Laguerre polynomials for problems on the half line \cite{doha2012a}. With the truncated series of smooth global functions, spectral methods are giving the solution of a particular problem \cite{parand1,parand2}. These methods,  with a relatively limited number of  degrees of freedom, provide such an accurate approximation for a smooth solution. Spectral coefficients tend to zero faster than any algebraic power of their index $n$ \cite{bhrawy3}.

Spectral methods can fall into 3 categories: Tau, Collocation and Galerkin  methods \cite{bhradoha}.
\begin{itemize}
\item The Tau spectral method is   used to approximate numerical solutions of various differential equations. This method considers   the solution as an expansion of  orthogonal polynomials/functions. Such coefficients, in this expansion, are set to  approximate the solution correctly \cite{bhrawy7}.
 
\item Collocation method  helps obtain a highly accurate solutions to nonlinear/linear differential equations \cite{bhrawy4,tal1,tal2,bhrawy5}.  Two common steps in collocation methods are: First,  suitable nodes (Gauss/Gauss-Radau/Gauss-Lobatto) are selected to restate a finite or discrete form of the differential equations.\\ 
Second,  a system of algebraic equations from the discretization of the original equation is achieved \cite{bhrawy6,parand3,parand4}.

\item 
In Galerkin Spectral method, trail and test functions are chosen the same \cite{boyd}; This method can result in a highly accurate solution.
\end{itemize}

It is said that spectral Galerkin methods are similar to Tau methods where in approximation by Tau method, differential equation is enforced \cite{bhrawy3}. 

Furthermore, some other numerical methods  like Finite difference method (FDM) and Finite element method (FEM)    need  network construction of data and they  perform locally. Although spectral methods are continuous and globally performing, they do not require network construction of data. 

As well as spectral methods, PS methods have also attracted researchers recently \cite{bhrawy8,bhrawy9,bhrawy10}. As mentioned previously, Sabeh et al. \cite{sabeh} investigated a PS method to solve optimal control problem.  PS methods are also utilized in the solution of other optimal control problems as well \cite{moh,fahroo,garg,shamsi}. These methods become popular because of their computational feasibility and efficiency.  In fact, in standard PS methods, interpolation operators are used to reducing the cost of computation of the inner product we encounter in some of the spectral methods. For this purpose, a set of distinct interpolation points $\{x_i\}_{i=0}^{n}$  is considered by which the corresponding Lagrange interpolants are achieved. Besides, when applying collocation points, $\{x_i\}_{i=0}^{n}$, the residual function is set to  vanish on the same set of points. Notwithstanding, the collocation points do not need to be chosen the same as the interpolation nodes; Indeed, just for having  the Kronecker delta property, they are considered  to be the same: as a consequence,  this property helps reduce computational cost noticeably as well \cite{delkhosh}. There are such  authors that utilized PS methods for the solution of optimal control as well. William  \cite{william} introduced a Jacobi PS method for solving an optimal control problem. He reported  that significant differences in computation time can be seen  for different parameters of the Jacobi polynomial. Garge et al. \cite{divya} presented a unified framework  for the numerical solution of optimal control problems using collocation at Legendre-Gauss (LG), Legendre-Gauss-Radau (LGR), and Legendre-Gauss-Lobatto (LGL) points and discussed the advantages of each for solving optimal control problems.  Chebyshev PS method was utilized by Fahroo et al. \cite{fahroo} to provide an optimal solution for optimal control problem.

\subsection{The main aim of this paper}
To the best of our knowledge, the use of PS methods for solving optimal control problems has been limited in the literature to either Chebyshev or Legendre methods. Noteworthy, the  PS method based on Jacobi can encompass a  wide range of other PS methods since the Legendre and Chebyshev nodes can be obtained as particular cases of the general Jacobi. This happens  when by  changing the parameters in the Jacobi polynomial a proper selection of the Jacobi parameters succeed in more accurate  real-time solutions to nonlinear optimal control problems. Meanwhile,  an arbitrary and not precise selection of nodes may result in a  poor interpolation characteristics such as the Runge phenomenon, therefore, the nodes in PS methods are selected as the Gauss-Radau points  \cite{william}.  

In this paper, we present a general formulation and a suitable numerical method called the GLJGR collocation method to solve 2DOCP for a class of distributed systems. The developed method is exponentially accurate and obtained by generalization of the classical Lagrangian polynomials. Additionally, the equation of the dynamics of optimal control problem is reformed as  a partial differential equation. 

This paper is arranged as follows: In Section 2, we present some preliminaries and drive some tools for introducing GL function, GLJGR collocation method, and their relevant derivative matrices. In Section 3, we apply the GLJGR collocation method to the solution of the 2DOCP. Section 4 shows numerical examples to demonstrate the effectiveness of the proposed method. Also, a conclusion is given in Section 5.

\section{Preliminaries, Conventions  and Notations}
In this section, we review some necessary definitions and  relevant properties of Jacobi polynomials. In the next step, we introduce Generalized Lagrangian (GL) functions. Then, we state and prove the accuracy of GL functions and develop GLJGR collocation method. Finally, in term of GLJGR collocation method, we give  a formula that expresses the   derivative matrix of the mentioned functions.

\subsection{Some properties of Jacobi Polynomials}
A basic property of the Jacobi polynomials is that they are the eigenfunctions to a singular Sturm--Liouville problem. Jacobi polynomials are defined on $[-1,1]$ and are of high interest recently   \cite{bhrawyzaky2015,bhrawyAlzaidy,BhrawyDoha2015,BhrawyAHDohaEH2015,Doha2015pseudo}.  The following recurrence relation generates the Jacobi polynomials \cite{dohajacobi}:
 \begin{equation}
 P_{k+1}^{\alpha,\beta}(x)=(a_k^{\alpha,\beta}x-b_k^{\alpha,\beta})P_k^{\alpha,\beta}-c_k^{\alpha,\beta}P_{k-1}^{\alpha,\beta}(x),~k\geq 1
 \end{equation}
 \begin{equation}
 P_0^{\alpha,\beta}(x)=1,~P_1^{\alpha,\beta}(x)=\frac{1}{2}(\alpha+\beta+2)x+\frac{1}{2}(\alpha-\beta),
 \end{equation}
where
$$a_k^{\alpha,\beta}=\frac{(2k+\alpha+\beta+1)(2k+\alpha+\beta+2)}{2(k+1)(k+\alpha+\beta+1)},$$
$$b_k^{\alpha,\beta}=\frac{(\beta^2-\alpha^2)(2k+\alpha+\beta+1)}{2(k+1)(k+\alpha+\beta+1)(2k+\alpha+\beta)},$$ 
$$c_k^{\alpha,\beta}=\frac{(k+\beta)(k+\alpha)(2k+\alpha+\beta+2)}{(k+1)(k+\alpha+\beta+1)(2k+\alpha+\beta)},$$

The Jacobi polynomials are satisfying the following identities:
\begin{eqnarray}
P_{n}^{\alpha,\beta}(-x)=(-1)^nP_{n}^{\beta,\alpha}(x),\\ P_{n}^{\alpha,\beta}(-1)=\frac{(-1)^n\Gamma(n+\beta+1)}{n!\Gamma(\beta+1)},\\ P_{n}^{\alpha,\beta}(1)=\frac{\Gamma(n+\alpha+1)}{n!\Gamma(\alpha+1)},
\end{eqnarray}
\begin{equation}
\bigg(P_{n}^{\alpha,\beta}(x)\bigg)^{(m)}=2^{-m}\frac{\Gamma(m+n+\alpha+\beta+1)}{\Gamma(n+\alpha+\beta+1)}P_{n-m}^{\alpha+m,\beta+m}(x).
\end{equation}
and its weight function is  $w^{\alpha,\beta}(x)=(1-x)^{\alpha}(1+x)^{\beta}$.\\
Moreover, the Jacobi polynomials are orthogonal on $[-1,1]$:
\begin{equation}\nonumber
\int_{-1}^{1}{P_{n}^{\alpha,\beta}(x)P_{m}^{\alpha,\beta}(x)}w^{\alpha,\beta}(x)=\delta_{m,n}\gamma_n^{\alpha,\beta},
\end{equation}
\begin{equation}
\gamma_n^{\alpha,\beta}=\frac{2^{\alpha+\beta+1}\Gamma(n+\alpha+1)\Gamma(n+\beta+1)}{(2n+\alpha+\beta+1)\Gamma(n+1)\Gamma(n+\alpha+\beta+1)},
\end{equation}
where $\delta_{m,n}$ is the Kronecker delta function. We define the weighted space $L_{w^{\alpha,\beta}}^2[-1,1]$.  The inner product and the norm of   $L_{w^{\alpha,\beta}}^2[-1,1]$ with respect to the weight function are defined as:
$$(g,h)_{w^{\alpha,\beta}}=\int_{-1}^{1}g(x)h(x){w^{\alpha,\beta}(x)}dx,~\| g\|_{w^{\alpha,\beta}}=(g,g)^{\frac{1}{2}}_{w^{\alpha,\beta}}$$

It is noted that the set of Jacobi polynomials forms a complete $L_{w^{\alpha,\beta}}^2[-1,1]$ system.

\subsection{Generalized Lagrangian (GL) functions}\label{GL}
In this section, generally,  the GL functions are introduced and the  suitable formulas for the first- and second-order derivative matrices of these functions are presented.  
\begin{definition}\label{general lagrange fomula definition}
Let $w(x)=\prod_{i=0}^N\big(u(x)-u(x_i)\big)$, then, the generalized Lagrange (GL) formula is shown as \cite{delkhosh} 
\begin{equation}\label{general lagrange fomula}
L_j^u(x)=\frac{w(x)}{ (u-u_j)\partial_xw(x_j)}=\frac{u'_jw(x)}{(u-u_j)\partial_u w(x_j)}=\kappa_j\frac{w(x)}{(u-u_j)},
\end{equation}
where $\kappa_j=\frac{u'_j}{\partial_u w(x_j)}$, and $u(x)$ is a continuous, arbitrary and sufficiently differentiable function, and $\partial_uw(x)=\frac{1}{u'}\partial_xw(x) $.

For simplicity $u=u(x)$ and $u_i=u(x_i)$ are considered.
\end{definition}

\begin{theorem}\label{derivative theorm}
Considering the GL functions $L_j^u(x)$  in Eq. (\ref{general lagrange fomula}), one can exhibit the first-order derivative matrices of GL functions as
$$D^{(1)}=[d_{kj}]\in \Re^{(n+1)\times(n+1)},~~~ 0\leq j,k \leq n,$$
where
\[   
d_{kj} = 
     \begin{cases}
      \kappa_j\frac{\partial_xw(x_j)}{u_k-u_j} ,&j\neq k, \\
      \\
       \kappa_j\frac{u'_j\partial_x^2 w(x_j)-u''_j\partial_xw(x_j)}{2{u'_j}^2}, & j=k.\\
     \end{cases}
\]
\end{theorem}

\begin{proof}: 
As the GL functions defined in Eq. (\ref{general lagrange fomula}), the first-order derivative formula for the case $k \neq j$ can be achieved as follows:
\begin{equation}
d_{kj}=\partial_xL_j^u(x_k)=\lim_{x\rightarrow x_k}\frac{L_j^u(x)-L_j^u(x_k)}{x-x_k}=\kappa_j\frac{\partial_xw(x_k)}{u_k-u_j}.
\end{equation}  
But, when $k=j$, with Hopital's rule: 
$$d_{jj}=\partial_xL_j^u(x_j)=\lim_{x\rightarrow x_j}\kappa_j\frac{(u-u_j)\partial_xw(x)-u'_jw(x)}{(u-u_j)^2}\overset{\mathrm{H}}{=}$$ 
$$\lim_{x\rightarrow x_j}\kappa_j\frac{(u-u_j)\partial_x^2w(x)-u''w(x)}{2u'(u-u_j)}\overset{\mathrm{H}}{=}\kappa_j\frac{u'_j\partial_x^2w(x_j)-u''_j\partial_xw(x_j)}{2{u'_j}^2}.$$
This completes the proof. $~\blacksquare$
\end{proof}

\begin{theorem}
Let $D^{(1)}$ be the above matrix (first order derivative matrix of GL functions) and define matrix $Q$ such that $Q=Diag(u'_0,u'_1,...,u'_N)$, $Q^{(1)}=Diag(u''_0,u''_1,...,u''_N)$, then,  the second-order  derivative matrix  of  GL functions can be formulated as: 
\begin{equation}
D^{(2)}=(Q^{(1)}+QD^{(1)})Q^{-1}D^{(1)}.
\end{equation}   
\end{theorem}
\begin{proof}:
See Ref. \cite{delkhosh}. $\blacksquare$
\end{proof}
For simplicity, from now on $D^{(1)}$ is considered $D$. 

\subsection{Generalized Lagrangian Jacobi Gauss-Radau (GLJGR) collocation method }\label{GLJGR section}
It is a well-established fact that a proper choice of collocation points is crucial in terms of accuracy and computational stability of the approximation by Lagrangian basis \cite{sabeh}. As  a good choice of such collocation points, we can refer to the well-known Gauss-Radau points in  which points lie inside (a,b) and one point is clustered near the endpoints. In this sequel, we use Jacobi-Gauss-Radau nodes.

In case of GLJGR collocation method, $w(x)$ in Eq. (\ref{general lagrange fomula}) can be considered as two approaches:  
$$w(x)=\lambda P_{n}^{\alpha,\beta+1}(u)(u-u_n),$$
and
$$w(x)=\lambda (u-u_0)P_{n}^{\alpha,\beta+1}(u),$$
where $\lambda$  is a real constant. As a matter of simplification, we write
\begin{equation}\label{J}
G(u)=P_{n}^{\alpha,\beta+1}(u),
\end{equation}
with the following important properties:
\begin{equation}
\partial_x G(u)=u'\frac{\Gamma(\alpha+\beta+n+3)}{2\Gamma(\alpha+\beta+n+2)}P_{n-1}^{\alpha+1,\beta+2}(u),
\end{equation} 
\begin{equation}
\partial_x^2 G(u)=\frac{\Gamma(\alpha+\beta+n+3)}{4\Gamma(\alpha+\beta+n+2)}\bigg(2u''P_{n-1}^{\alpha+1,\beta+2}(u)+(u')^2(\alpha+\beta+n+3)P_{n-2}^{\alpha+2,\beta+3}(u)\bigg),
\end{equation} 

Lets speak of the first approach. Assume
$$w(x)=\lambda P_{n}^{\alpha,\beta+1}(u)(u-u_n),$$
then, we have:
\begin{equation}
\partial_x w(x)=\lambda \big[u'G(u)+(u-u_n)\partial_xG(u)~\big],
\end{equation}
\begin{eqnarray}\label{W}
&\partial_x^2 w(x)=\lambda \big[u''G(u)+2u'\partial_x G(u)+(u-u_n)\partial_x^2G(u)\big].
\end{eqnarray}
\begin{equation}\label{tsd}
G(u_n)=\frac{\Gamma(\alpha+1+n)}{(n)!\Gamma(\alpha+1)},~~~
\partial_x G(u_n)=u'_n\frac{\Gamma(\alpha+\beta+n+3)}{2\Gamma(\alpha+\beta+n+2)}\frac{\Gamma(\alpha+1+n)}{(n-1)!\Gamma(\alpha+2)}.
\end{equation} 
 
Recalling that $\{P_{n}^{\alpha,\beta+1}(u_j)=0\}_{j=0}^{n-1}$ and using formulas in Eqs. (\ref{J})--(\ref{tsd}), we find the entry of the first-order derivative matrix of GL functions as:
\[   
d_{kj} = 
\begin{cases}
\frac{u'_kP_{n-1}^{\alpha+1,\beta+2}(u_k)(n!)(\alpha+\beta+n+2)\Gamma(\alpha+1)}{2\Gamma(\alpha+1+n)} ,&j=n, 0\leq k\leq n-1, \\
\frac{2u_k'\Gamma(\alpha+n+1)}{\Gamma(\alpha+1)(n!)(\alpha+\beta+n+2)P_{n-1}^{\alpha+1,\beta+2}(u_j)(u_k-u_j)(u_j-u_k)}, & k=n, 0\leq j \leq n-1,\\
\frac{u_k'(u_k-u_n)}{(u_j-u_n)(u_k-u_j)}\frac{P_{n-1}^{\alpha+1,\beta+2}(u_k)}{P_{n-1}^{\alpha+1,\beta+2}(u_j)}    ,& 0\leq j\neq k\leq n-1,  \\
\frac{u_j'}{u_j-u_n}+u_j'(\alpha+\beta+n+3)\frac{P_{n-3}^{\alpha+3,\beta+3}(u_j)}{4P_{n-2}^{\alpha+2,\beta+2}(u_j)},&0\leq k=j\leq n-1,\\
\frac{u_k' (n-1)(\alpha+\beta+n+2)}{2(\alpha+1) },&   ~~k=j=n ,\\
\end{cases}
\]

Similar to this fasion, for the second approach one can write:
$$w(x)=\lambda (u-u_0)P_{n}^{\alpha,\beta+1}(u),$$
\begin{equation}
\partial_x w(x)=\lambda \big[u'G(u)+(u-u_0)\partial_xG(u)~\big],
\end{equation}
\begin{eqnarray}\label{W}
&\partial_x^2 w(x)=\lambda \big[u''G(u)+2u'\partial_x G(u)+(u-u_0)\partial_x^2G(u)\big].
\end{eqnarray} 
\begin{equation}
G(u_0)=\frac{\Gamma(\beta+2+n)(-1)^n}{(n)!\Gamma(\beta+2)},~~~
\partial_x G(u_0)=u'_0\frac{\Gamma(\alpha+\beta+n+3)}{2\Gamma(\alpha+\beta+n+2)}\frac{(-1)^{n-1}\Gamma(\beta+2+n)}{(n-1)!\Gamma(\beta+3)}.
\end{equation} 
 
Now this is obvious that $\{P_{n}^{\alpha,\beta+1}(u_j)=0\}_{j=1}^{n}$.   Therefore, by the second approach, the entries of defined matrix $D$ can be filled as:   
\[   
d_{kj} = 
\begin{cases}
\frac{u'_kP_{n-1}^{\alpha+1,\beta+2}(u_k)(n!)(\alpha+\beta+n+2)\Gamma(\beta+2)}{2\Gamma(\beta+2+n)(-1)^n} ,&j=0, 1\leq k\leq n, \\
\frac{2u_k'\Gamma(\beta +n+2)(-1)^n}{\Gamma(\beta+2)(n!)(\alpha+\beta+n+2)P_{n-1}^{\alpha+1,\beta+2}(u_j)(u_k-u_j)(u_j-u_k)}, & k=0, 1\leq j \leq n,\\
\frac{u_k'(u_k-u_0)}{(u_j-u_0)(u_k-u_j)}\frac{P_{n-1}^{\alpha+1,\beta+2}(u_k)}{P_{n-1}^{\alpha+1,\beta+2}(u_j)}    ,& 1\leq j\neq k\leq n,  \\
\frac{u_j'}{u_j-u_0}+u_j'(\alpha+\beta+n+3)\frac{P_{n-3}^{\alpha+3,\beta+3}(u_j)}{4P_{n-2}^{\alpha+2,\beta+2}(u_j)},&1\leq k=j\leq n,\\
\frac{-u_k' (n)(\alpha+\beta+n+2)}{2(\beta+2) },&   ~~k=j=0 ,\\ 
\end{cases}
\]

More specifically,  Legendre,  Chebyshev, and  ultraspherical polynomials can be obtained as special cases from the proposed method. These cases are summarized in the following corollaries:\\
\textbf{Corollary 1:} If $\alpha=\beta$, we have the  all the mentioned formulas of GL functions, $D$, $D^{(2)}$ for Gegenbauer (ultraspherical) polynomials (symmetric Jacobi polynomials).\\
\textbf{Corollary 2:} If $\alpha=\beta=0$, we have the  all the mentioned formulas of GL functions, $D$, $D^{(2)}$ for Legendre case.\\
\textbf{Corollary 3:} If $\alpha=\beta=-0.5$, we have  all the mentioned formulas of GL functions, $D$, $D^{(2)}$ for  Chebyshev case (the 1st kind). \\
\textbf{Corollary 4:} If $\alpha=\beta=0.5$, we have  all the mentioned formulas of GL functions, $D$, $D^{(2)}$ for  Chebyshev case (the 2nd kind).\\
\textbf{Corollary 5:} If $\alpha=-0.5$, $\beta=0.5$, we have  all the mentioned formulas of GL functions, $D$, $D^{(2)}$ for  Chebyshev case (the 3rd kind).\\
\textbf{Corollary 6:} If $\alpha=0.5$, $\beta=-0.5$, we have  all the mentioned formulas of GL functions, $D$, $D^{(2)}$ for  Chebyshev case (the 4-th kind).

\subsection{Operational matrix of GL functions}
Defining the one-column vectors 
\[
\hat{L}_n^u(x)=\begin{bmatrix}
  L_0^u(x)\\ L_1^u(x)\\ \vdots\\ \L_n^u(x)\\
  \end{bmatrix},
\hat{H}_n^u(x)=\begin{bmatrix}
  \partial_xL_0^u(x)\\ \partial_xL_1^u(x)\\ \vdots\\\partial_xL_n^u(x)\\
  \end{bmatrix},
  C=\begin{bmatrix}
  c_0\\ c_1\\ \vdots\\ c_n\\
  \end{bmatrix},\]
for approximation of a function $\xi(x)$ can write
$$\xi(x)=\sum_{i=0}^{n}c_iL_i^u(x)=C^T\hat{L}_n^u(x),$$
and similarly for the derivative of this function can rewrite it as 
\begin{equation}
\partial_x\xi(x)=C^TK\hat{L}_i^u(x),
\end{equation}
where $K\in \Re^{(n+1)\times(n+1)}$ is the operational matrix of derivative where
$$K\hat{L}_n^u(x)=\hat{H}_n^u(x)$$
and in other words,  
\[
\begin{bmatrix}
  K_{00}&K_{01}&\dots&K_{0n}\\
  K_{10}&K_{11}&\dots&K_{1n}\\
  \vdots\\
   K_{n0}&K_{n1}&\dots&K_{nn}
  \end{bmatrix}\begin{bmatrix}
  L_0^u(x)\\L_1^u(x)\\ \vdots\\L_n^u(x)\\
  \end{bmatrix} =\begin{bmatrix}
  \partial_xL_0^u(x)\\ \partial_xL_1^u(x)\\ \vdots\\\partial_xL_n^u(x)\\
  \end{bmatrix}.
\]
 Taking the $j$-th row 
\[
\begin{bmatrix}
  K_{j0}&K_{j1}&\dots&K_{jn}
  \end{bmatrix}\begin{bmatrix}
  L_0^u(x)\\L_1^u(x)\\ \vdots\\L_n^u(x)\\
  \end{bmatrix} = \partial_xL_j^u(x)  ,
\]
by collocating the  Jacobi-Gauss-Radau nodes ($\{x_i\}_{i=0}^{n}$)   in  this equation, we obtain:
$$K_{ji}=\partial_xL_j^u(x_i),$$
$$K_{ji}=d_{ij}.$$
This means that the operational matrix for these functions are 
\begin{equation}\label{k1}
K=D^{T},
\end{equation}
where $D$ is defined in Section \ref{GLJGR section}. Similarly, for  $\partial_{xx}\xi(x)$ can say 
\begin{equation}\label{k2}
\partial_{xx}\xi(x)=C^T\big(D^{(2)}\big)^{T}\hat{L}_i^u(x),
\end{equation} 
where $D^{(2)}$ is defined and found in Section \ref{GLJGR section}.


\section{Numerical Method}
The main objective of this section is to develop the GLJGR collocation method to solve 2DOCP. In this section, firstly,  a promising function approximation method has been presented. Then, GLJGR collocation is implemented so as to accomplish the introduction of the presented numerical method  for 2DOCP of interest.
\subsection{Function approximation}
If $H=L^2(\eta)$, $\eta=0\cup (0,1)\times (0,1)\cup 1$, where
$$s=\{L_i^u(x)L_j^u(t)~ | ~0\leq i\leq n, 0\leq j\leq m\}$$
$s\in H $, $s$ is the set of GL functions product and 
$$V_{nm}=Span\{L_i^u(x)L_j^u(t) ~| ~0\leq i \leq n ,~0\leq j \leq m\},$$ 
where $V_{nm}$ is a finite dimensional vector space. For any $q\in S$, one can find the best approximation of $q$ in space $V_{nm}$ as $P_{nm}(x,t)$ such that
$$\exists P_{nm}(x,t) \in V_{nm}, \forall l_{nm}(x,t)\in V_{nm}, \| q-P_{nm}(x,t)\|_2 \leq \| q-l_{nm}(x,t)\|_2$$
Therefore, for any $P_{nm}(x,t)\in V_{nm}$ can write
\begin{equation}
P_{nm}(x,t)\simeq \sum_{i=0}^{n}\sum_{j=0}^{m} c_{ij}L_i^u(x)L_j^u(t)=(\hat{L}_n^u(x))^TC\hat{L}_m^u(t)
\end{equation}
in which  $C$ is a  matrix  of $\Re^{(m+1)\times (n+1)}$ and $c_{ij}$ are the relevant coefficients.
$L_i^u(x)$ is considered by the first approach mentioned in subsection \ref{GLJGR section} in which $u(x)=\frac{2}{R}x-1$, and $L_j^u(t)$ is based on the second approach and considered as $u(t)=2t-1$.


\subsection{Implementation of GLJGR collocation method for solving the 2DOCP}
Now, for approximation of state and control functions 
\begin{equation}\label{u}
y(x,t)\simeq y_{nm}(x,t)=\sum_{i=0}^{n}\sum_{j=0}^{m} a_{ij}L_i^u(x)L_j^u(t)=(\hat{L}_n^u(x))^TA\hat{L}_m^u(t)
\end{equation}
\begin{equation}\label{z}
z(x,t)\simeq z_{nm}(x,t)=\sum_{i=0}^{n}\sum_{j=0}^{m} b_{ij}L_i^u(x)L_j^u(t)=(\hat{L}_n^u(x))^TB\hat{L}_m^u(t)
\end{equation}
where
  \[
A=\begin{bmatrix}
a_{00}&a_{01}&\dots&a_{0m}\\
a_{10}&a_{11}&\dots&a_{1m}\\
\vdots\\
a_{n0}&a_{n1}&\dots&a_{nm}
\end{bmatrix},
B=\begin{bmatrix}
b_{00}&b_{01}&\dots&b_{0m}\\
b_{10}&b_{11}&\dots&b_{1m}\\
\vdots\\
b_{n0}&b_{n1}&\dots&b_{nm}
\end{bmatrix},
\hat{L}_m^u(t)=\begin{bmatrix}
L_0^u(t)\\L_1^u(t)\\ \vdots\\L_m^u(t)\\
\end{bmatrix},
\hat{L}_n^u(x)=\begin{bmatrix}
L_0^u(x)\\L_1^u(x)\\ \vdots\\L_n^u(x)\\
\end{bmatrix}    
\]
We define residual functions $res(x,t)$ by substituting Eqs. (\ref{u}) and (\ref{z}) in Eq. (\ref{subjecto}) 
 \begin{equation}\label{res_nm}
 res(x,t)=-\frac{x\partial z_{nm}(x,t)}{\partial t}+k(\frac{x\partial^2 z_{nm}(x,t)}{\partial x^2}+r\frac{\partial z_{nm}(x,t)}{\partial t})+xy_{nm}(x,t),
\end{equation}
in terms of Eqs. (\ref{k1}) and  (\ref{k2}) one can read
\begin{equation}
\frac{\partial z_{nm}(x,t)}{\partial x}=\bigg(D^T\hat{L}_n^u(x)\bigg)^TB\hat{L}_m^u(t)=\hat{L}_n^u(x)^TDB\hat{L}_m^u(t),
\end{equation}
\begin{equation}
\frac{\partial^2 z_{nm}(x,t)}{\partial x^2}=\bigg((D^{(2)})^T\hat{L}_n^u(x)\bigg)^TB\hat{L}_m^u(t)=\hat{L}_n^u(x)^TD^{(2)}B\hat{L}_m^u(t),
\end{equation}  
\begin{equation}
\frac{\partial z_{nm}(x,t)}{\partial t}=\hat{L}_n^u(x)^TB\hat{D}^T\hat{L}_m^u(t),
\end{equation}
$A,B\in \Re^{(n+1)\times (m+1)} $,
$D\in \Re^{(n+1)\times (n+1)} $ and
$\hat{D}\in \Re^{(m+1)\times (m+1)} $
so the Eq. (\ref{res_nm}) can be restated as 
\begin{eqnarray}\label{be}
&& res(x,t)=-x\hat{L}_n^u(x)^TB\hat{D}^T\hat{L}_m^u(t)+k\bigg(x\hat{L}_n^u(x)^TD^{(2)}B\hat{L}_m^u(t)\nonumber\\&&+r\hat{L}_n^u(x)^TB\hat{D}^T\hat{L}_m^u(t)\bigg)+x\hat{L}_n^u(x)^TA\hat{L}_m^u(t),
\end{eqnarray}
and the initial and boundary conditions of the problem are obtained as
\begin{equation}
z_{nm}(x,0)\simeq z_0(x), 0<x<R,~~~~~~~z_{nm}(R,t)\simeq 0,t>0,
\end{equation}
and within the  assumption of Eq. (\ref{z}) and Eq. (\ref{u})
\begin{equation}\label{condition_nm}
\hat{L}_n^u(x)^TB\hat{L}_m^u(0)\simeq z_0(x),~~\hat{L}_n^u(R)^TB\hat{L}_m^u(t)\simeq 0.
\end{equation}
As $0$ and $ R $ are the  $ t_0 $ and $ x_n $, with the aid of the characteristic of Lagrange polynomials, we write
\begin{equation}\label{intitlagra}
  \hat{L}_n^u(x)^TB
    \begin{bmatrix}
	1\\0\\ \vdots\\0\\
\end{bmatrix} =\hat{L}_n^u(x)^TB[:,0],~~~   
 \begin{bmatrix}
 0&0&...&1
 \end{bmatrix}B\hat{L}_m^u(t)=B[n,:]\hat{L}_m^u(t),
\end{equation}
where $ B[:,0] $ and $B[n,:]$ are the first column and $n$-th row of Matrix $B$.
 
Then, with $n+1$ collocation nodes in $x$ space and  $m+1$ collocation nodes in $t$ space,  a set of  algebraic equations is constructed by using Eq. (\ref{be}) together with the conditions in Eq. (\ref{condition_nm}) ( simplification of Eq. (\ref{intitlagra}) is also used).
\begin{equation}\label{dastga}
 \begin{cases}
 F[i]=\hat{L}_n^u(x_i)^TB[:,0]- z_0(x_i)=B_{i0}- z_0(x_i), ~i=0,\dots,n,\\
F[n+j]=B[n,:]\hat{L}_m^u(t_j)=B_{nj}, ~j=1,\dots,m,\\
F[n+m+(i)(m+1)+j+1]=res(x_i,t_j),~i=0,...,n,~j=0,...,m,\\
 \end{cases}
 \end{equation}
in which  $ res(x_i,t_j) $ can be considered as
 \begin{equation}\nonumber
 res(x_i,t_j)=-x_iB[i,:]\hat{D}^T[:,j]+k\bigg(x_iD^{(2)}[i,:]B[:,j]+rB[i,:]\hat{D}^T[:,j]\bigg)+x_ia_{ij}.
\end{equation}
or
 \begin{equation}
 res(x_i,t_j)=(kr-x_i)B[i,:]\hat{D}^T[:,j]+kx_iD^{(2)}[i,:]B[:,j]+x_ia_{ij}.
\end{equation}
The reason why in the second case of (\ref{dastga}) $j$ is started from 1 is that: in both examples we will consider later, in the first case of (\ref{dastga}) we have $z_0(x_n=R)=0$ and this makes a redundancy with the second case of (\ref{dastga}) at $j=0$. To avoid this, $j$ is started from 1.  

In what follows, we have used $x_i$ and $t_j$ which are considered as the following statements. Firstly,
\begin{equation}
\{x_i\}_{i=0}^{n}=\{x|(P_{n}^{\alpha,\beta+1}(u)(u-u_n))=0\},
\end{equation}
as mentioned $x\in (0,R]$, then $u(x)=\frac{2x}{R}-1$. Therefore,
\begin{equation}
\bigg(\{x_i\}_{i=0}^{n-1}=\{x|P_{n}^{\alpha,\beta+1}(\frac{2x}{R}-1)=0\}\bigg)\cup\{ x_n=R\}.
\end{equation}
One can see that this assignment is based on the first approach in Section \ref{GLJGR section}.\\
With the same fashion, for variable $t$ we have
\begin{equation}
\{t_j\}_{j=0}^{m}=\{t|(u-u_0)P_{n}^{\alpha,\beta+1}(u)=0\}
\end{equation}
where $t \in [0,1)$ and $u(t)=2t-1$.
\begin{equation}
\{ t_0=0\}\cup\bigg(\{t_j\}_{j=1}^{m}= \{t|P_{n}^{\alpha,\beta+1}(2t-1)=0\}\bigg)
\end{equation}
Again, one can consider  this assignment  based on the second approach in Section \ref{GLJGR section}.\\
At the next step we approximate the integral existing in the 2DOCP. For this, we exploit Gauss Jacobi quadratures.\\
For estimating an integral by Gauss Jacobi quadratures, we do as follow: 
\begin{equation*}
\int_a^bf(v)w^{\alpha,\beta}(\frac{2(v-a)}{b-a}-1)dv\simeq(\frac{b-a}{2})\int_{-1}^1f(\frac{b-a}{2}s+\frac{b+a}{2})w^{\alpha,\beta}(s)ds,
\end{equation*}
\begin{equation}\label{numerical_integration}
=\frac{b-a}{2}\sum_{i=0}^{N}f(\frac{b-a}{2}s_i+\frac{b+a}{2})\varpi_i=\frac{b-a}{2}\sum_{i=0}^{N}f(v_i)\varpi_i,
\end{equation}
where $v\in [a,b]$ and $s\in [-1,1]$
, and Eq. (\ref{numerical_integration}) is exact when $degree \big(f(v)\big)\leq 2N+1$ for Gauss Jacobi quadratures.\\
$\{s_i\}_{i=0}^{N}$ are Gauss Jacobi nodes and their relevant weights $\{\varpi\}_{i=0}^{N}$ are \cite{shen}
\begin{equation}
\varpi_i=\frac{\Gamma(N+\alpha+2)\Gamma(N+\beta+2)}{\bigg(P_{N+1}^{\alpha,\beta}(s_i)\bigg)'(1-s_i^2)}\frac{2^{\alpha+\beta+1}}{\Gamma(N+2+\alpha+\beta)\Gamma(N+2)},
\end{equation}
The cost functional $J$ is estimated by a numerical integration method. For this, we applied  Gauss-Jacobi quadratures in Eq. (\ref{numerical_integration}) for both variables $t$ and  $x$.
\begin{equation*}\label{min_nm}
min ~~ J=\frac{1}{2}\int_0^1\int_0^R x^r\big(c_1z_{nm}^2(x,t)+c_2y_{nm}^2(x,t)\big)dxdt,
\end{equation*}
\begin{equation}\label{takhmin_integral}
\simeq \frac{R}{8}\bigg(c_1\sum_{i=0}^{N}\sum_{j=0}^{M}\frac{\varpi_x^i {(\frac{\hat{x_i}+1}{2})}^r \varpi_t^j z_{nm}^2(\frac{\hat{x_i}+1}{2},\frac{\hat{t_j}+1}{2})}{w^{\alpha,\beta}(\hat{x_i})w^{\alpha,\beta}(\hat{t_i})}+c_2\sum_{i=0}^{N}\sum_{j=0}^{M}\frac{\varpi_x^i {(\frac{\hat{x_i}+1}{2})}^r \varpi_t^j y_{nm}^2(\frac{\hat{x_i}+1}{2},\frac{\hat{t_j}+1}{2})}{w^{\alpha,\beta}(\hat{x_i})w^{\alpha,\beta}(\hat{t_i})}\bigg)
\end{equation}
where $\varpi_x^i$, $\varpi_t^j$, $\hat{x_i}$, $\hat{t_j}$ are as follows
\begin{equation}
\{\hat{x_i}\}_{i=0}^{N}=\{x|P_{N+1}^{\alpha,\beta}(x)=0\},
\end{equation}
\begin{equation}
\{\hat{t_j}\}_{j=0}^{M}=\{t|P_{M+1}^{\alpha,\beta}(t)=0\},
\end{equation}
\begin{equation}
\varpi_x^i=\frac{\Gamma(N+\alpha+2)\Gamma(N+\beta+2)}{\bigg(P_{N+1}^{\alpha,\beta}(\hat{x_i})\bigg)'(1-\hat{x_i}^2)}\frac{2^{\alpha+\beta+1}}{\Gamma(N+2+\alpha+\beta)\Gamma(N+2)},
\end{equation}
\begin{equation}\label{hdd}
\varpi_t^j=\frac{\Gamma(M+\alpha+2)\Gamma(M+\beta+2)}{\bigg(P_{M+1}^{\alpha,\beta}(\hat{t_j})\bigg)'(1-\hat{t_j}^2)}\frac{2^{\alpha+\beta+1}}{\Gamma(M+2+\alpha+\beta)\Gamma(M+2)}.
\end{equation}
From Eq. (\ref{takhmin_integral}) to Eq. (\ref{hdd}) for both variables $t$ and  $x$ we consider  $15$ Gauss Jacobi nodes ($M=N=14$).\\
Thus, on the basis of what we have just discussed, the optimal control problem is reduced to a parameter optimization problem. This can be stated as follows:
\begin{equation}
L(A,B,\hat{\lambda})=J+\sum_{i=0}^{2(n+m)+nm+1}\lambda_iF[i]
\end{equation}
$\hat{\lambda}=\{\lambda_i\}_{i=0}^{2(n+m)+nm+1}$ are Lagrange multipliers. Therefore, the   minimization problem can be under these new conditions
\[
   \begin{dcases}
\frac{\partial L}{\partial a_{ij}}=0,~~i=0,\dots,n,~j=0,\dots,m,\\
\frac{\partial L}{\partial b_{ij}}=0,~~i=0,\dots,n,~j=0,\dots,m,\\
\frac{\partial L}{\partial \lambda_{i}}=0,~~i=0,\dots,2(n+m)+nm+1,
   \end{dcases}
\]
produce a system of $4(n+m+1)+3nm$  algebraic equations which can be solved by a mathematical software for achieving the unknowns. The solution of this system is given by using Maple software.

\section{Numerical Examples}
In this section,  the GLJGR collocation method is used to solve  2 cases of aforementioned 2DOCP in Eq. (\ref{minproblem}). The aim is to find the  state and control functions $z(x,t)$, $y(x,t)$ that minimize the cost function $J$.
\begin{example}\label{ex1}
\end{example}
Consider the forementioned  2DOCP with $c_1=c_2=r=R=k=1$ and initial condition \cite{meme,ozmir} $$z(x,0)=z_0(x)=1-\big(\frac{x}{R}\big)^2,~0<x<R.$$
So, the optimal control problem of Eq. (\ref{minproblem}) would be
\begin{equation}
min ~~ J=\frac{1}{2}\int_0^1\int_0^1 x\big(z^2(x,t)+y^2(x,t)\big)dxdt,
\end{equation}
subject to 
\begin{equation}
x\frac{\partial z}{\partial t}=x(\frac{\partial^2 z}{\partial x^2}+\frac{1}{x}\frac{\partial z}{\partial t})+xy(x,t),
\end{equation}
with boundary and initial conditions
\begin{equation}
z(x,0)=z_0(x)=1-x^2, 0<x<1,~~~~~~~z(1,t)=0,t>0.
\end{equation}

Consider the assumptions  mentioned in Example  \ref{ex1}. with the methodology presented in Section 2, 3, we approximate the function $z(x,t)$ and $y(x,t)$. In Table \ref{table:ex1:overall} the presented method is used to solve the 2DOCP of Example  \ref{ex1}. This table, by presenting the value of cost functional, simply shows the accuracy of the presented method for different choices of $n$ and $m$. The effect of $\alpha$ and $\beta$ shown as well to provide Chebyshev (all four kinds), Legendre cases and other different cases. As the aim of this paper is to find $z(x,t)$, $y(x,t)$ in order to minimize $J$, we plotted these state and control functions in Fig. \ref{fig:ex1}(a,c). Also, the graph of these function for  different  number of basis is illustrate in Fig. \ref{fig:ex1}(b,d); they  are the surface plots of the state and control functions. The comparison with Ritz method by Mamehrashi \cite{meme} has been made in Table \ref{table:ex1:comparison}. These results show that the presented method provides a more accurate solution. Similar to what has been concluded in \cite{meme}, we resulted that the control and state functions initially have distinct values over the $x$ axis and as time goes by they tend to reach the same value: This phenomenon is representative of a diffusion process. 

\begin{table}[!ht]
\centering
\caption{Numerical result of Example  \ref{ex1}.}
\label{table:ex1:overall}
\scalebox{0.85}{
\begin{tabular}{ccccccccccccccc}
\hline
$n$&$m$&$\alpha$&$\beta$&$J$&$n$&$m$&$\alpha$&$\beta$&$J$&$n$&$m$&$\alpha$&$\beta$&$J$\\
\hline
\multirow{10}{*}{2} & \multirow{10}{*}{2} & 0 & 0 & 0.0248814 & \multirow{10}{*}{5} & \multirow{10}{*}{5} & 0 & 0 & 0.01391215 & \multirow{10}{*}{5} & \multirow{10}{*}{7} & 0 & 0 & 0.0089648 \\
 &  & -0.5 & -0.5 & 0.0284523 &  &  & -0.5 & -0.5 & 0.01567798 &  &  & -0.5 & -0.5 & 0.0095673 \\
 &  & 0.5 & 0.5 & 0.0230089 &  &  & 0.5 & 0.5 & 0.01248992 &  &  & 0.5 & 0.5 & 0.0084049 \\ 
 &  & -0.5 & 0.5 & 0.0278755 &  &  & -0.5 & 0.5 & 0.01794318 &  &  & -0.5 & 0.5 & 0.0105544 \\ 
 &  & 0.5 & -0.5 & 0.0224106 &  &  & 0.5 & -0.5 & 0.01128294 &  &  & 0.5 & -0.5 & 0.0079716 \\  
 &  & 0 & 1 & 0.0255469 &  &  & 0 & 1 & 0.01516258 &  &  & 0 & 1 & 0.0092312 \\ 
 &  & 1 & 0 & 0.0209952 &  &  & 1 & 0 & 0.01069205 &  &  & 1 & 0 & 0.0078231 \\ 
 &  & 0 & 2 & 0.0265474 &  &  & 0 & 2 & 0.01506972 &  &  & 0 & 2 & 0.0084796 \\  
 &  & 2 & 0 & 0.0205804 &  &  & 2 & 0 & 0.00913921 &  &  & 2 & 0 & 0.0074379 \\  
 &  & 3 & 1 & 0.0193088 &  &  & 3 & 1 & 0.00859934 &  &  & 3 & 1 & 0.0071598 \\ \hline
\multirow{10}{*}{7} & \multirow{10}{*}{5} & 0 & 0 & 0.0186896 & \multirow{10}{*}{7} & \multirow{10}{*}{10} & 0 & 0 & 0.00813737 & \multirow{10}{*}{10} & \multirow{10}{*}{10} & 0 & 0 & 0.0089628 \\ 
 &  & -0.5 & -0.5 & 0.0207502 &  &  & -0.5 & -0.5 & 0.00862000 &  &  & -0.5 & -0.5 & 0.0095609 \\ 
 &  & 0.5 & 0.5 & 0.0168244 &  &  & 0.5 & 0.5 & 0.00765917 &  &  & 0.5 & 0.5 & 0.0083101 \\ 
 &  & -0.5 & 0.5 & 0.0223249 &  &  & -0.5 & 0.5 & 0.00777906 &  &  & -0.5 & 0.5 & 0.0091915 \\ 
 &  & 0.5 & -0.5 & 0.0163430 &  &  & 0.5 & -0.5 & 0.00850356 &  &  & 0.5 & -0.5 & 0.0087234 \\  
 &  & 0 & 1 & 0.0195875 &  &  & 0 & 1 & 0.00765102 &  &  & 0 & 1 & 0.0082872 \\  
 &  & 1 & 0 & 0.0151638 &  &  & 1 & 0 & 0.00757913 &  &  & 1 & 0 & 0.0083260 \\ 
 &  & 0 & 2 & 0.0189639 &  &  & 0 & 2 & 0.00652118 &  &  & 0 & 2 & 0.0067052 \\ 
 &  & 2 & 0 & 0.0128480 &  &  & 2 & 0 & 0.00733578 &  &  & 2 & 0 & 0.0079058 \\ 
 &  & 3 & 1 & 0.0114265 &  &  & 3 & 1 & 0.00699366 &  &  & 3 & 1 & 0.0072780 \\ \hline
\end{tabular}}
\end{table}

\begin{figure}[!ht]
    \begin{subfigure}[b]{0.45\textwidth}
        \centering
        \includegraphics[width=\textwidth]{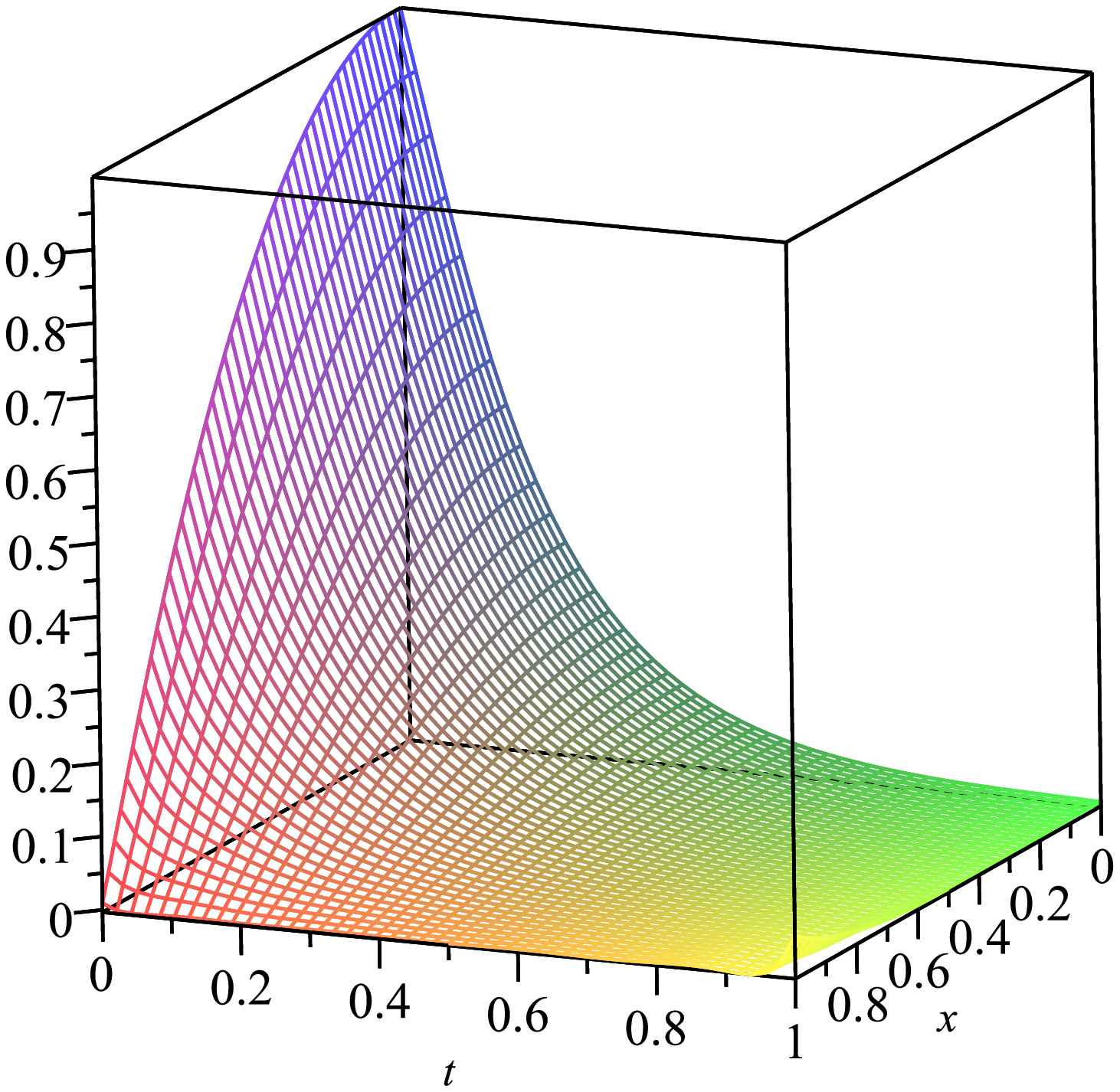}
        \caption{Approximation of state function $z(x, t)$ for Example  \ref{ex1}}
        \label{fig:b}
    \end{subfigure}
     \hfill
     \begin{subfigure}[b]{0.45\textwidth}
        \centering
        \includegraphics[width=\textwidth]{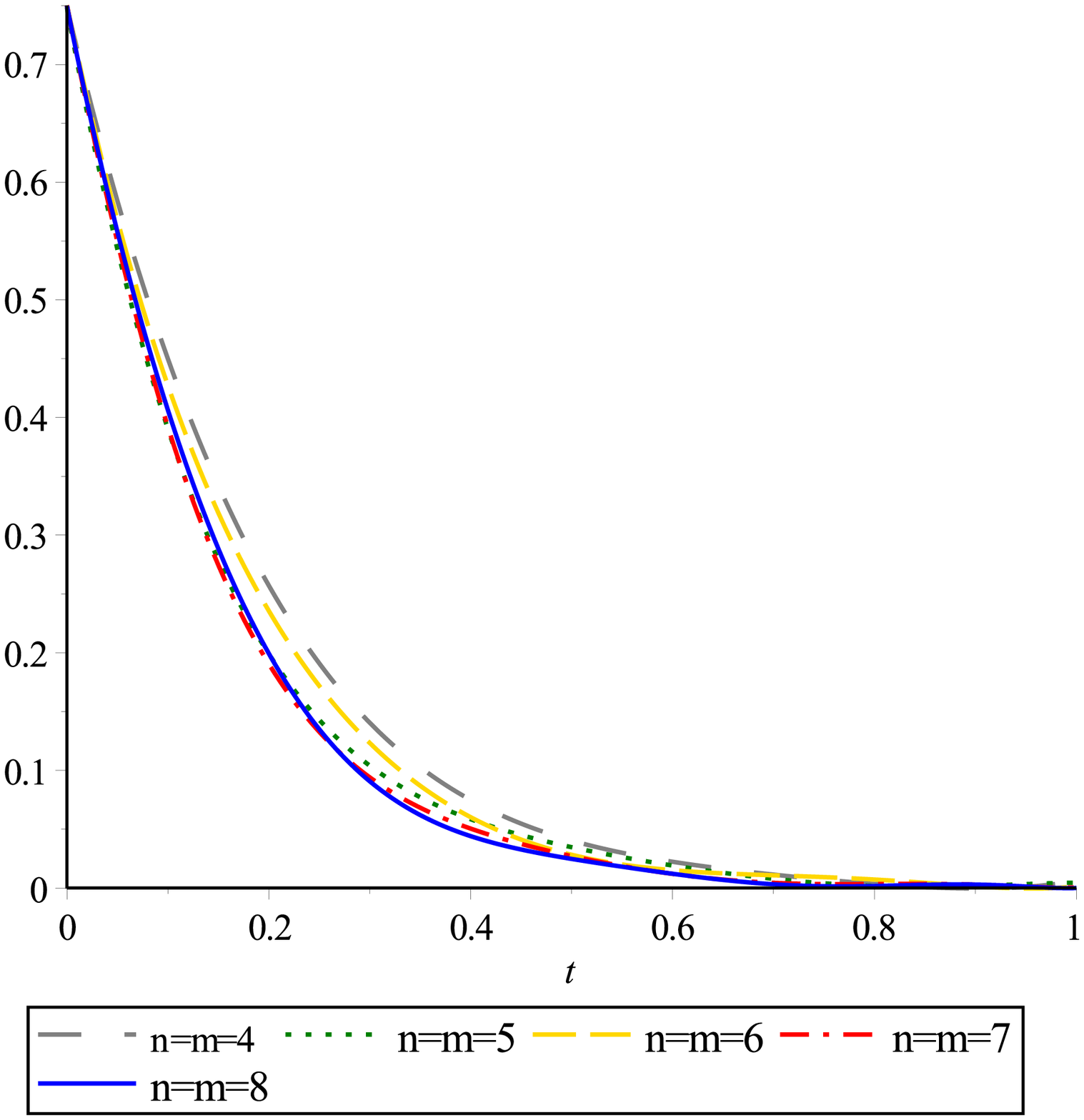}
        \caption{Approximation  of $z(x, t)$ for $x = 0.5$ with various choices of  $m, n$}
        \label{fig:c}
    \end{subfigure}
    \hfill
    \begin{subfigure}[b]{0.45\textwidth}
        \centering
        \includegraphics[width=\textwidth]{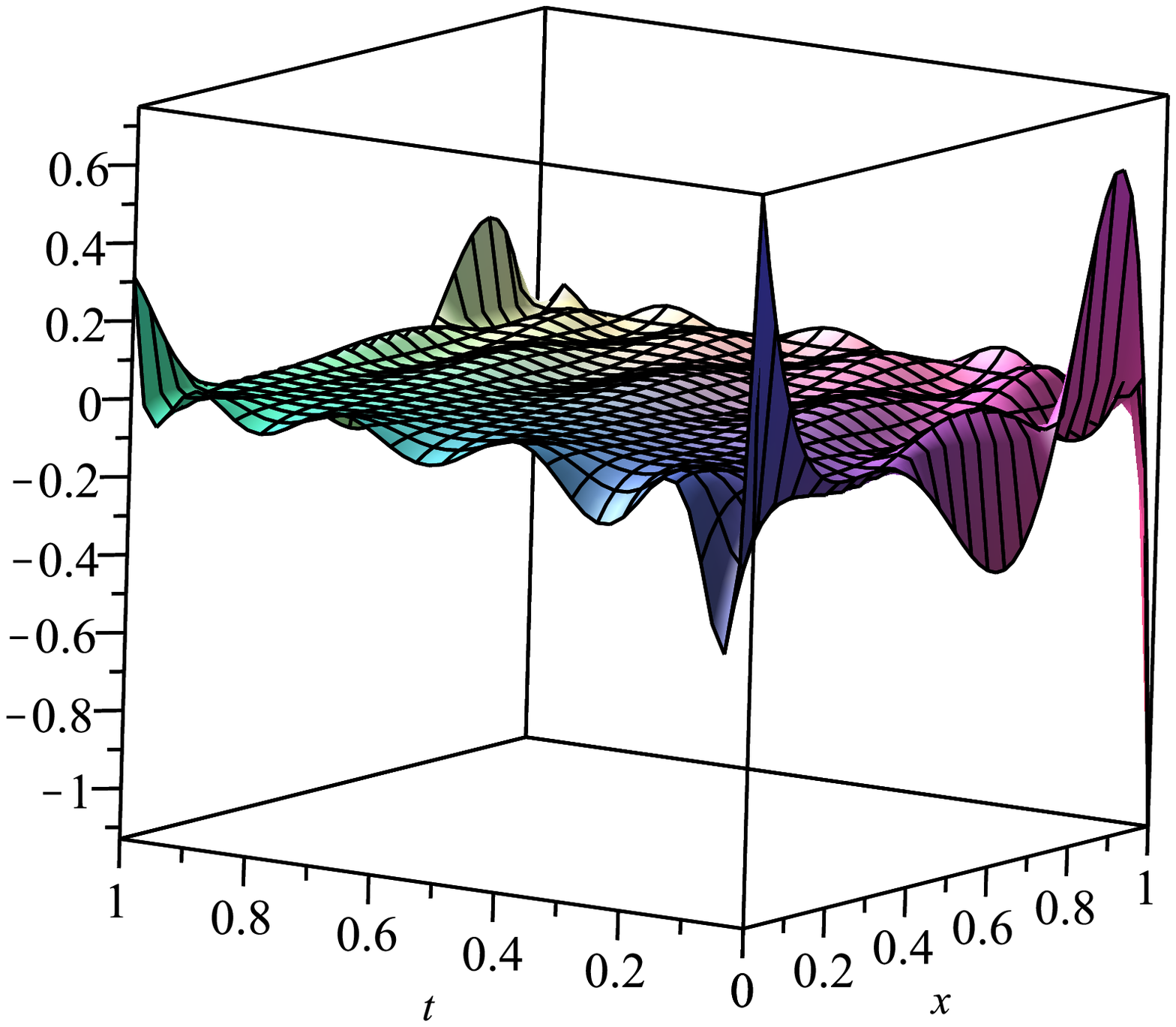}
        \caption{Approximation of control variable $y(x, t)$ for Example  \ref{ex1}}
        \label{fig:b}
    \end{subfigure}
     \hfill
     \begin{subfigure}[b]{0.45\textwidth}
        \centering
        \includegraphics[width=\textwidth]{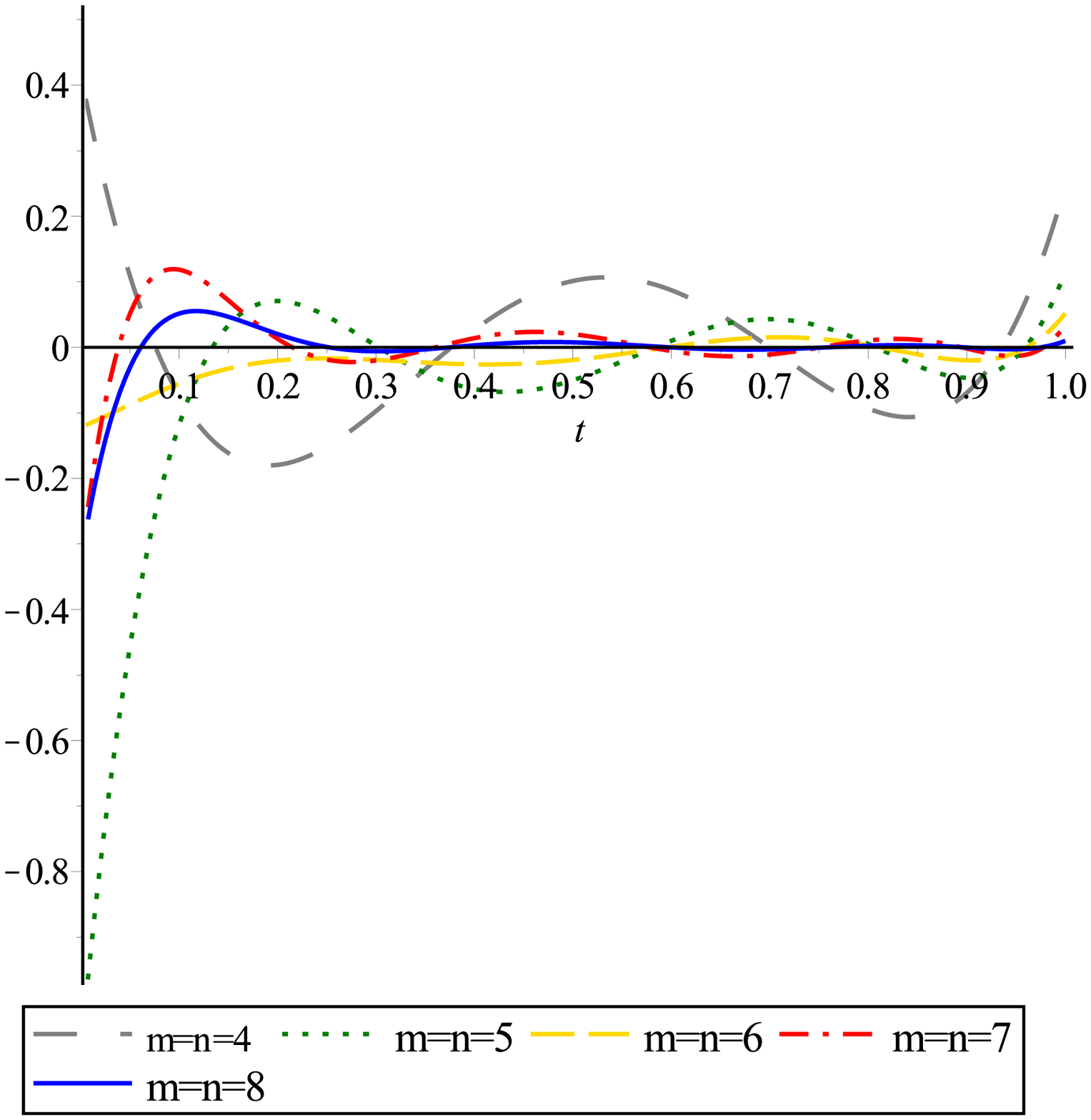}
        \caption{Approximation  of $y(x, t)$ for $x = 0.5$ with various choices of  $m, n$}
        \label{fig:c}
    \end{subfigure}
     \hfill
    \caption{Plots of approximate state and contorl functions for Example  \ref{ex1}.}\label{fig:ex1}
\end{figure}

\begin{table}[!ht]
\centering
\caption{Numerical result for $J$ in Example  \ref{ex1} with presented method and Ritz method \cite{meme} . (
$\alpha=0$,$\beta=2$)}
\label{table:ex1:comparison}
\scalebox{0.85}{
\begin{tabular}{llllllll}
\hline
n                & 1          & 2          & 2          & 2          & 3          & 3          & 3          \\ 
m                & 4          & 4          & 6          & 7          & 7          & 9          & 10         \\ \hline
Ritz method \cite{meme}      & 0.081044   & 0.028790   & 0.018283   & 0.016484   & 0.013027   & 0.010405   & 0.007569   \\ \hline
presented method & 0.01586417 & 0.01408943 & 0.01396815 & 0.01401548 & 0.00473073 & 0.00470629 & 0.00468442 \\ \hline
\end{tabular}}
\end{table}

\begin{example}\label{ex2}
\end{example}
In this example the $r=2$, $c_1=c_2=R=k=1$ and the initial condition is \cite{meme,mehdihasan} $$z(x,0)=z_0(x)=sin(2\pi x),~0<x<R.$$ 
Therefore, we have 
\begin{equation}
min ~~ J=\frac{1}{2}\int_0^1\int_0^1 x^2\big(z^2(x,t)+y^2(x,t)\big)dxdt,
\end{equation}
subject to 
\begin{equation}
x\frac{\partial z}{\partial t}=x(\frac{\partial^2 z}{\partial x^2}+\frac{2}{x}\frac{\partial z}{\partial t})+xy(x,t),
\end{equation}
with boundary and initial conditions
\begin{equation}
z(x,0)=z_0(x)=sin(2\pi x), 0<x<1,~~~~~~~z(1,t)=0,t>0.
\end{equation}

Like Example  \ref{ex1}, the method discussed in Section 2, 3, is utilized to approximate the solution of Example  \ref{ex2}. The cost functional $J$ for different selection of $n$, $m$, $\alpha$, $\beta$ is calculated in Table \ref{table:ex2:overall}.  Figure \ref{fig:ex2}(a,c) show numerical results for state and control functions i.e. $z(x,t)$ and $y(x,t)$, respectively. These results are plotted  in Fig. \ref{fig:ex2}(b,d) at $x=0.5$ in  a surface plot. Note that initially the state and control at two different locations differ, but as the time progresses, the two values become very close. As said in Example  \ref{ex1}. this is because of diffusion. A comparison with Ritz method \cite{meme} is made and reported in Table \ref{table:ex2:comparison}. The results in this table demonstrate that the presented method is  more accurate and reliable.
\begin{table}[!ht]
\centering
\caption{Numerical result of Example  \ref{ex2}.}
\label{table:ex2:overall}
\scalebox{0.85}{
\begin{tabular}{ccccccccccccccc}
\hline
$n$&$m$&$\alpha$&$\beta$&$J$&$n$&$m$&$\alpha$&$\beta$&$J$&$n$&$m$&$\alpha$&$\beta$&$J$\\
\hline
\multirow{10}{*}{2} & \multirow{10}{*}{2} & 0 & 0 & 0.981774266 & \multirow{10}{*}{5} & \multirow{10}{*}{5} & 0 & 0 & 2.36052795 & \multirow{10}{*}{5} & \multirow{10}{*}{7} & 0 & 0 & 0.7709361 \\ 
 &  & -0.5 & -0.5 & 0.598751916 &  &  & -0.5 & -0.5 & 2.56788894 &  &  & -0.5 & -0.5 & 0.8706653 \\ 
 &  & 0.5 & 0.5 & 1.226311717 &  &  & 0.5 & 0.5 & 2.15726157 &  &  & 0.5 & 0.5 & 0.8083347 \\ 
 &  & -0.5 & 0.5 & 0.135957972 &  &  & -0.5 & 0.5 & 2.26726831 &  &  & -0.5 & 0.5 & 0.6988001 \\ 
 &  & 0.5 & -0.5 & 2.079513828 &  &  & 0.5 & -0.5 & 2.38904227 &  &  & 0.5 & -0.5 & 0.8083347 \\ 
 &  & 0 & 1 & 0.383652016 &  &  & 0 & 1 & 2.04317188 &  &  & 0 & 1 & 0.5890023 \\ 
 &  & 1 & 0 & 2.022399132 &  &  & 1 & 0 & 2.22743042 &  &  & 1 & 0 & 0.7332970 \\ 
 &  & 0 & 2 & 0.101961113 &  &  & 0 & 2 & 1.68004117 &  &  & 0 & 2 & 0.4074085 \\ 
 &  & 2 & 0 & 2.012714897 &  &  & 2 & 0 & 2.08193333 &  &  & 2 & 0 & 0.6923438 \\ 
 &  & 3 & 1 & 1.877493818 &  &  & 3 & 1 & 1.77464161 &  &  & 3 & 1 & 0.5390153 \\ \hline
\multirow{10}{*}{7} & \multirow{10}{*}{5} & 0 & 0 & 2.498801296 & \multirow{10}{*}{7} & \multirow{10}{*}{10} & 0 & 0 & 0.16882398 & \multirow{10}{*}{10} & \multirow{10}{*}{10} & 0 & 0 & 0.1877156 \\
 &  & -0.5 & -0.5 & 2.650340836 &  &  & -0.5 & -0.5 & 0.20321134 &  &  & -0.5 & -0.5 & 0.2208658 \\ 
 &  & 0.5 & 0.5 & 2.305646950 &  &  & 0.5 & 0.5 & 0.13032446 &  &  & 0.5 & 0.5 & 0.1469488 \\ 
 &  & -0.5 & 0.5 & 2.415879813 &  &  & -0.5 & 0.5 & 0.14219532 &  &  & -0.5 & 0.5 & 0.1510822 \\
 &  & 0.5 & -0.5 & 2.499141221 &  &  & 0.5 & -0.5 & 0.18137586 &  &  & 0.5 & -0.5 & 0.2050793 \\
 &  & 0 & 1 & 2.201622390 &  &  & 0 & 1 & 0.10295491 &  &  & 0 & 1 & 0.1096786 \\ 
 &  & 1 & 0 & 2.349064681 &  &  & 1 & 0 & 0.15196838 &  &  & 1 & 0 & 0.1759308 \\ 
 &  & 0 & 2 & 1.837451911 &  &  & 0 & 2 & 0.05137739 &  &  & 0 & 2 & 0.0481434 \\ 
 &  & 2 & 0 & 2.162014225 &  &  & 2 & 0 & 0.13443163 &  &  & 2 & 0 & 0.1631296 \\ 
 &  & 3 & 1 & 1.831793809 &  &  & 3 & 1 & 0.08179826 &  &  & 3 & 1 & 0.1001137 \\ \hline
\end{tabular}}
\end{table}

\begin{figure}[!ht]
    
    \begin{subfigure}[b]{0.45\textwidth}
        \centering
        \includegraphics[width=\textwidth]{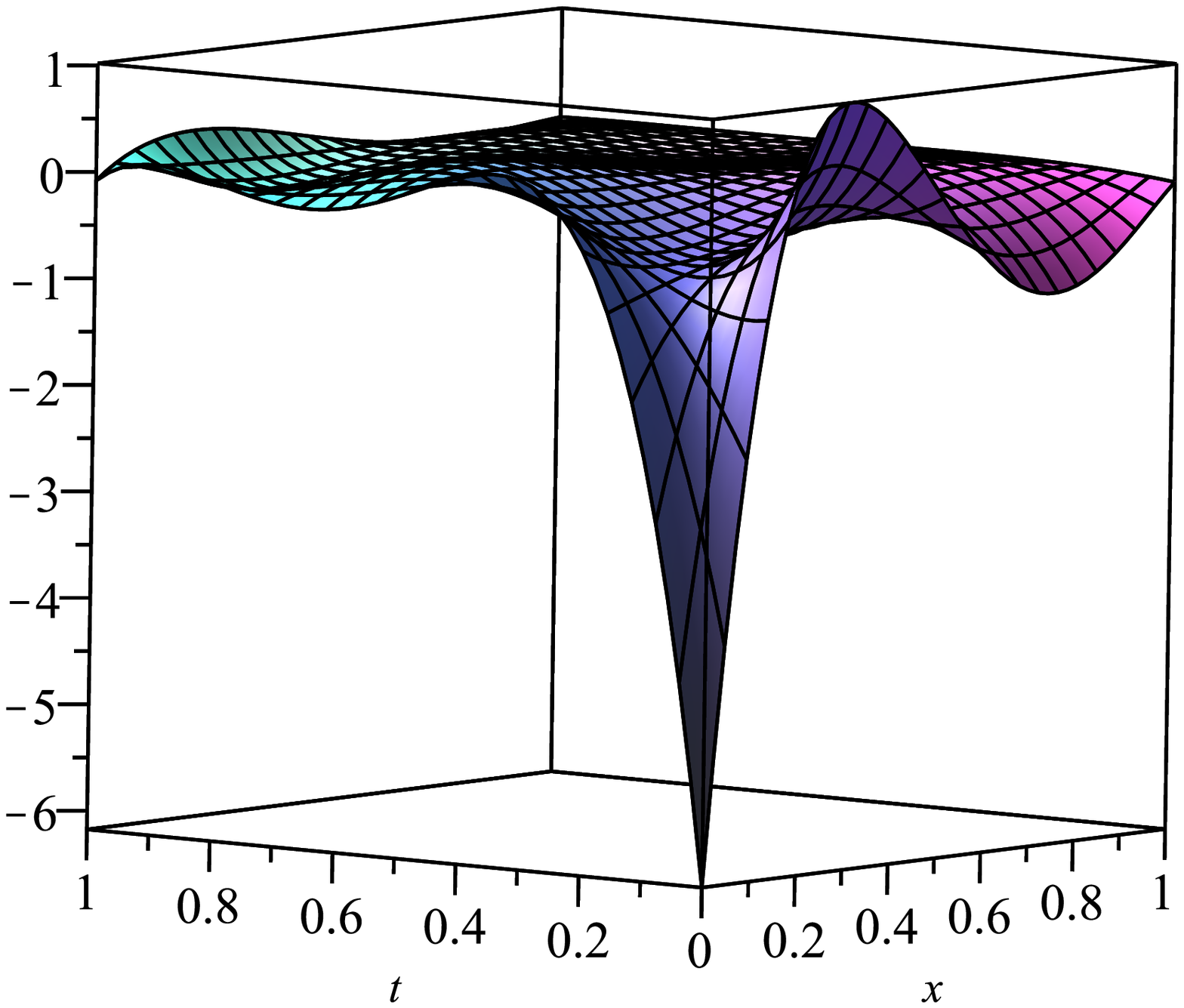}
        \caption{Approximation of state function $z(x, t)$ for Example  \ref{ex2}}
        \label{fig:b}
    \end{subfigure}
     \hfill
     \begin{subfigure}[b]{0.45\textwidth}
        \centering
        \includegraphics[width=\textwidth]{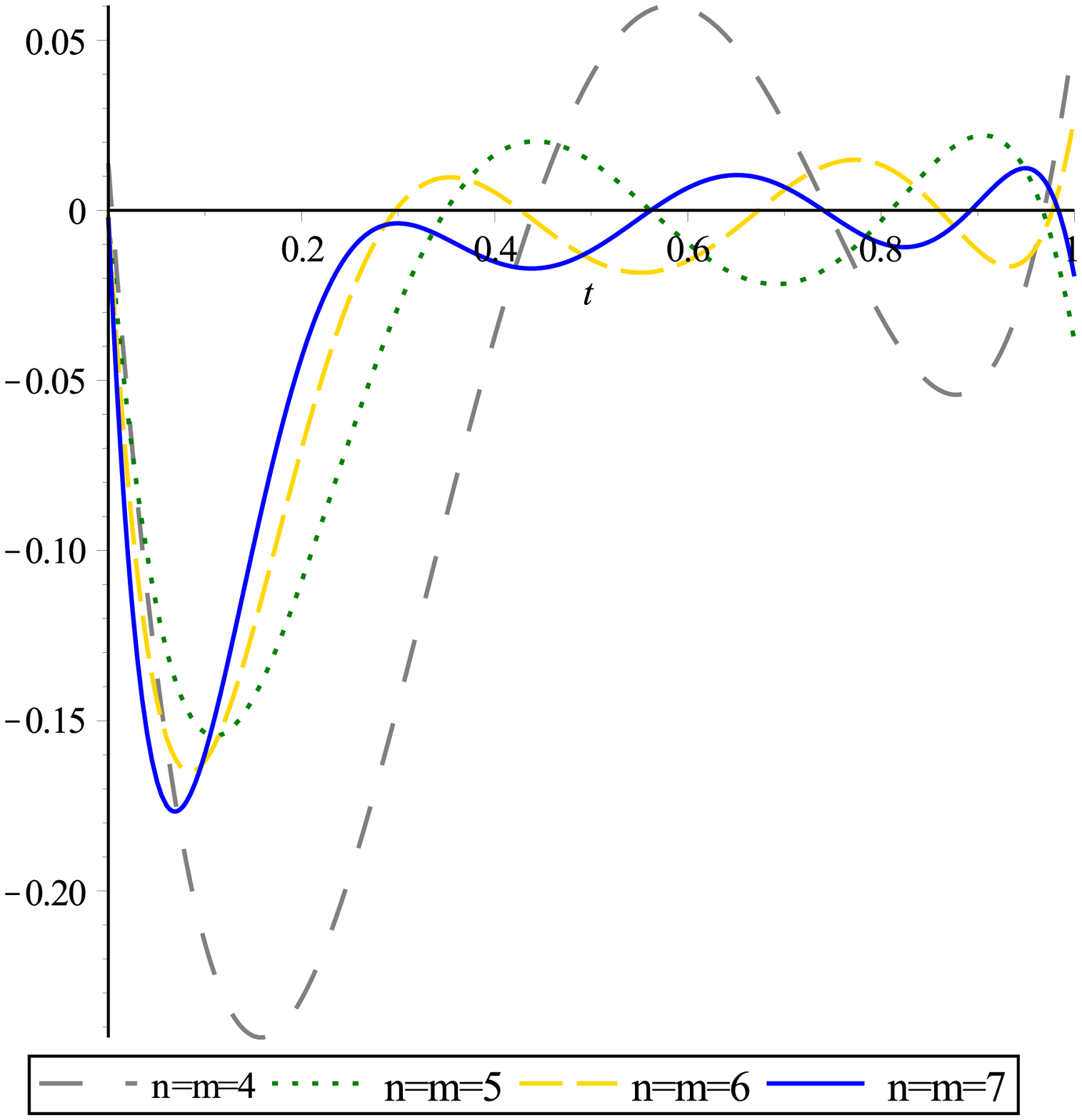}
        \caption{Approximation  of $z(x, t)$ for $x = 0.5$ with various choices of  $m, n$}
        \label{fig:c}
    \end{subfigure}
    \hfill
    \begin{subfigure}[b]{0.45\textwidth}
        \centering
        \includegraphics[width=\textwidth]{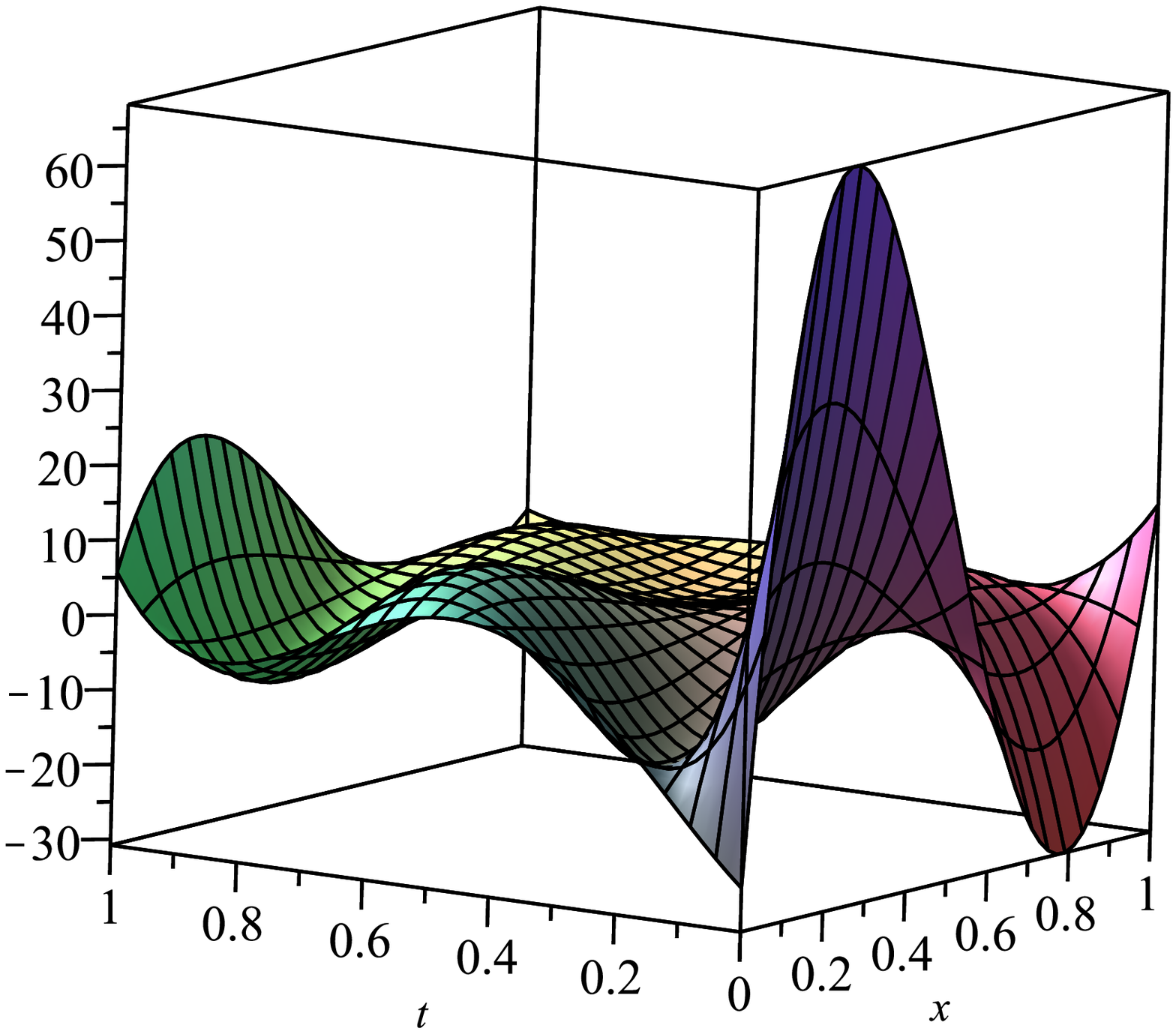}
        \caption{Approximation of control function $y(x, t)$ for Example  \ref{ex2}}
        \label{fig:b}
    \end{subfigure}
     \hfill
     \begin{subfigure}[b]{0.45\textwidth}
        \centering
        \includegraphics[width=\textwidth]{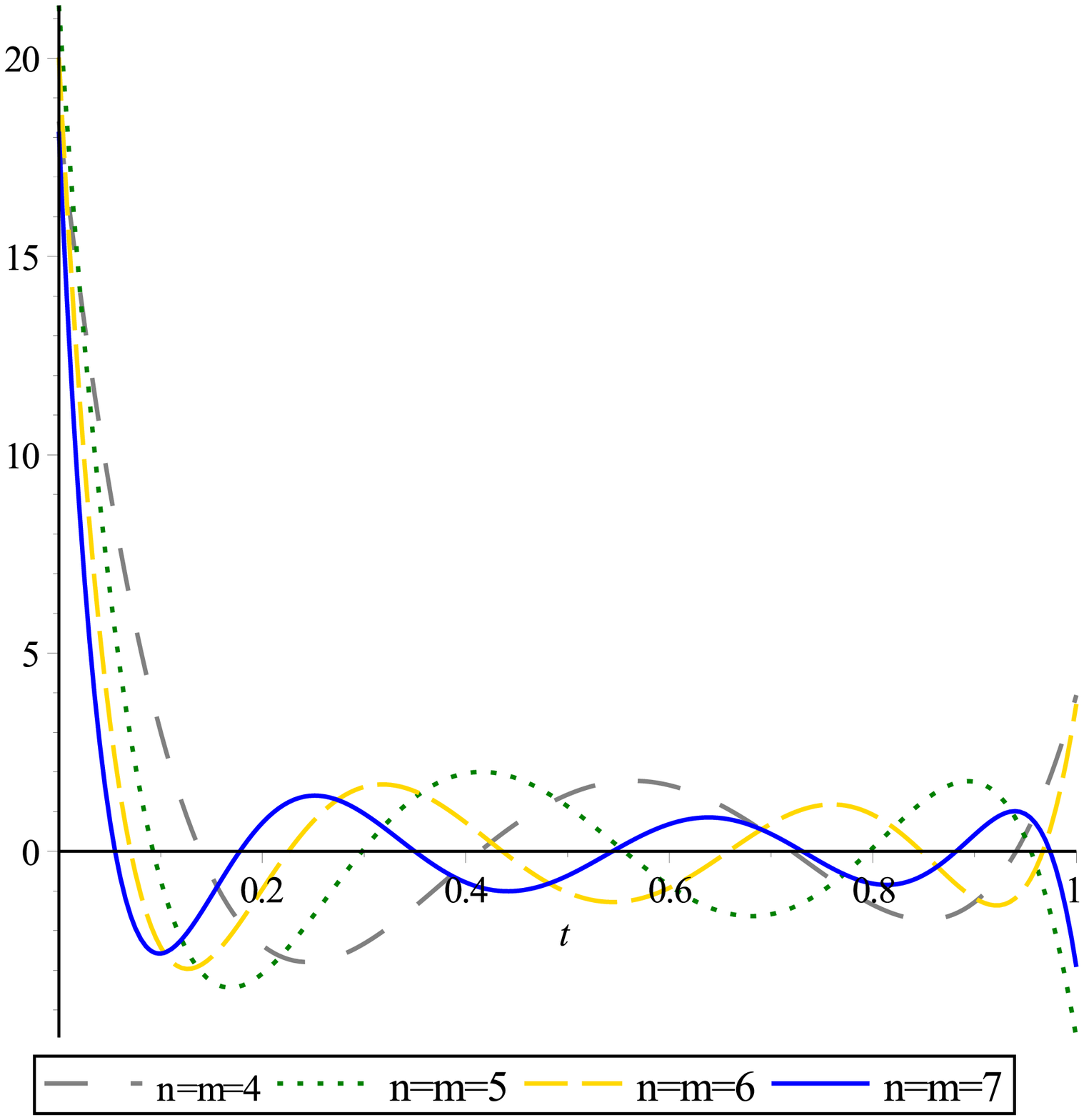}
        \caption{Approximation  of $y(x, t)$ for $x = 0.5$ with various choices of  $m, n$}
        \label{fig:c}
    \end{subfigure}
     \hfill
    \caption{Plots of approximate state and contorl functions for Example  \ref{ex2}.}\label{fig:ex2}
\end{figure}
\begin{table}[!ht]
\centering
\caption{Numerical result for $J$ in Example  \ref{ex2} with presented method and Ritz method \cite{meme} . (
$\alpha=0$,$\beta=2$)}
\label{table:ex2:comparison}
\scalebox{0.85}{
\begin{tabular}{llllllll}
\hline
n                & 4           & 5           & 5           & 6           & 6           & 7           & 7           \\
m                & 5           & 5           & 6           & 6           & 7           & 7           & 8           \\ \hline
Ritz method \cite{meme}           & 2.72722     & 1.92027     & 1.27424     & 0.91850     & 0.55287     & 0.54935     & 0.36868     \\ \hline
presented method & 1.834166023 & 1.680041170 & 0.834822493 & 0.865965461 & 0.426441006 & 0.467294422 & 0.229834210 \\ \hline
\end{tabular}}
\end{table}

\section{Conclusion}
In this study, a 2DOCP is investigated. This problem has beneficial applications in many chemical, biological,  and physical fields of studies. The goal of this  article is to develop an efficient and accurate method to solve this nonlinear 2DOCP. The method is based upon GLJGR collocation method.  Firstly,  the GL functions introduced so as to satisfy in delta Kronecker function and GLJGR collocation method is described. As expressed, these functions are a generalization of the classical Lagrangian polynomials. The corresponding differentiation matrices of $D^{(1)}$ and $D^{(2)}$, can be obtained by simple formulas. The main advantage of this proposed formulas is  that  these formulas are derivative-free. Additionally, The accuracy of the presented method by GL function has exponential convergence rate. Secondly, the obtained results compared with Mamehrashi et al. \cite{meme} results, showing the accuracy and reliability of the presented method. By this comparison, we emphasized that comparing with Ritz method developed by Mamehrashi et al. \cite{meme}, the more satisfactory results obtained only in the same number of polynomials order. This numerical approach is applicable and effective for such  kind of nonlinear 2DOCPs and other problems that can be approximated by Gauss-Radau nodes.


\end{document}